\documentclass[a4paper,11pt,reqno, english]{amsart}  
\title{Counterexamples to the weak Hedetniemi Conjecture}
\author{Steven Raanes}
\let\Horig\H
\usepackage[utf8]{inputenc}
\usepackage[T1]{fontenc}
\usepackage{amsmath,amsthm}
\usepackage{amsfonts,amssymb,enumerate}
\usepackage{url,paralist}
\usepackage{mathtools}  
\usepackage[colorlinks=true,urlcolor=blue,linkcolor=red,citecolor=magenta]{hyperref}
\usepackage{enumerate}
\usepackage{tikz}
\usepackage[left=1in,right=1in,top=1in,bottom=1in]{geometry}
\usepackage{listings}
\usepackage[numbers]{natbib}

\theoremstyle{plain}
\newtheorem{thm}{Theorem}[section]
\newtheorem{lem}[thm]{Lemma}

\newtheorem*{thm*}{Theorem}

\theoremstyle{definition}
\newtheorem{defn}[thm]{Definition}

\newtheorem{rem}[thm]{Remark}

       \def\H{{\mathbb{H}}}      \def\N{{\mathbb{N}}}    \def\R{{\mathbb{R}}}        \def\Z{{\mathbb{Z}}}

\def\ep{{\varepsilon}}

\begin{document}
\begin{abstract}
    For any $n$, there exists finite simple graphs $G$ and $H$ such that $\chi(G),\chi(H) > n$ and $\chi(G \times H) = 4$.
\end{abstract}
\maketitle

\section{Introduction}

The tensor product $G \times H$ of graphs $G$ and $H$ is defined by $V(G \times H) = V(G) \times V(H)$ and $(g,h) \sim (g',h')$ if $g \sim g'$ and $h \sim h'$. The undirected Poljak-R\"odl function $f(n)$ is defined in \cite{poljak1981arc} by 
    $$f(n) = \text{min} \{ \chi(G \times H):\chi(G),\chi(H)\geq n \}$$

The famous Hedetniemi conjecture that $\chi(G \times H) = \text{min}\{ \chi(G),\chi(H)\}$ can be phrased as $f(n) = n$. The inequality $f(n) \leq n$ is simple, and El-Zahar and Sauer \cite{el1985chromatic} showed $f(n)= n$ for $n\leq 4$. In a breakthrough paper, Shitov \cite{shitov2019counterexamples} proved that $f(n) < n$ for $n$ large enough. Shitov's methods were then adapted by Tardif and Zhu \cite{tardif2019note} to show $f(n) \leq n - (\text{log}n)^{1/4 - o(1)}$. This result was improved by He and Wigderson \cite{he2021hedetniemi} to $f(n) \leq (1-10^{-9})n + o(n)$. Zhu \cite{zhu2020note} improved the bound to $f(n) \leq n/2 + o(n)$, and most recently Tardif \cite{tardif2023chromatic} sharped this to $f(n) \leq \lceil n/2\rceil + 3$ and showed that $f(5) = 4$. The conjecture that $f(n)$ is unbounded is called the weak Hedetneimi conjecture and was first asked in \cite{poljak1981arc}. We show this conjecture is false.

\begin{thm}
    For all $n \in \N$, $f(n) \leq 4.$
\end{thm}

Since El-Zahar and Sauer showed $f(n) = n$ for $n \leq 4$ and $f(n)$ is monotone, this bound is tight and in fact completely determines $f(n)$ for all $n$.
 
Our constructions are topological, so we first recall the definitions and theorems we will need.

\section{Preliminaries}

An \textit{abstract simplicial complex} K on vertex set $V(K)$ is a set system $K \subset 2^{V(K)}$ such that $\{v \} \in K$ for all $v \in V(K)$, and $\sigma \in K$, $\tau \subset \sigma$ implies $\tau \in  K$. Elements of $K$ are called simplices. The dimension of a simplex $\sigma \in K$ is $|\sigma|-1$ and the dimension of $K$ is the maximum dimension of any simplex in $K$. Let $K$
and $L$ be simplicial complexes. A map $f:V(K) \to V(L)$ is a \textit{simplicial map} if $f(\sigma) = \{f(v): v \in \sigma\}$ is a simplex in $L$ for all $\sigma \in K$. 

Define the \textit{standard $n-$simplex} to be the convex hull of the standard basis vectors in $\R^{n+1}$. The \textit{geometric realization} of an abstract simplicial complex $K$ is the topological space $|K|$ obtained by identifying each simplex of $K$ with the standard simplex of the same dimension and gluing along common faces. We identify a simplical complex with its geometric realization when no confusion will arise. A simplical map $f: K \to L$ induces a continuous map between geometric realizations by extending $f: V(K) \to V(L)$ linearly on simplices. A \textit{polyhedron} is a topological space homeomorphic to the geometric realization of a finite simplicial complex. 

Define the \textit{order complex} $\Delta P$ of a poset $P$ to be the simplicial complex with vertices the elements of $P$ and simplices the chains in $P$. An order preserving map $f:P \to Q$ induces a simplical map $f_*:\Delta P \to \Delta Q$ by $f_*(p) = f(p)$. The \textit{face poset} $\mathcal{F}(K)$ of a simplical complex $K$ is the poset with elements the non-empty simplices of $K$ and $\tau \leq \sigma$ if $\tau \subset \sigma$. Note $|\Delta \mathcal{F}(K)|$ is homeomorphic to $K$.

Both our lower and upper bounds on chromatic number rely on equivariant topology. A more leisurely introduction to equivariant methods and their applications to various combinatorial and geometric problems can be found in \cite{matouvsek2003using} or \cite{kozlov2008combinatorial}. 

An action of $G$ on a space $X$ is a group homomorphism $\rho:G \to \text{Homeo}(X)$ and we call a space $X$ with a $G$ action a $G$ space. We will denote $\rho(g)(x)$ by $g\cdot x$ or $gx$. We say $G$ acts freely on $X$ if for all $x \in X$, $gx = x$ implies that $g=1$. An action of $G$ on a polyhedron $K$ is simplicial if $\rho(g)$ is a simplicial map for all $g \in G$. We say $G$ acts without fixed points on $X$ if for all $x \in X$ there exists some $g \in G$ such that $gx \neq x$. If $G$ acts on $X$ and $Y$, a map $f:X \to Y$ is equivariant, denoted $f:X \to_G Y$, if $f(g\cdot x)  = g\cdot f(x)$ for all $g \in G$ and $x \in X$. The composition of equivariant maps is equivariant.

Much of the usefulness of equivariant topology comes from non-existence results for equivariant maps, such as the following result known as Dold's theorem. Recall a space $X$ is $n-$connected if $\pi_i(X)$ is trivial for $0 \leq i \leq n$, where $\pi_i(X)$ is the $i$th homotopy group of $X$.

\begin{thm} 
    $\textnormal{(Dold's theorem \cite{dold1983simple})}$  Let $G$ be a non-trivial finite group, $K$ an $n-$connected simplical complex with a simplical $G$ action, and $L$ a simplical complex of dimension at most $n$ with free simplical $G$ action. Then there is no continuous equivariant map from $K$ to $L$.
\end{thm}

 If $G$ acts freely on $X$ and $H$ act freely on $Y$, then $G \times H$ acts freely on $X \times Y$ by $(g,h)\cdot(x,y) = (gx,hy)$. Since we work with actions on product metric spaces in this paper, we need the following lemma.

\begin{lem}\label{metricproductaction}
    Let $X$ and $Y$ be metric spaces. Let $G$ act on $X$ by isometries and $H$ act on $Y$ by isometries. Then $G \times H$ acts by isometries on $X \times Y$ by $(g,h)\cdot(x,y) = (gx,hy)$ when $X \times Y$ is equipped with the max metric, $d((x_1,y_1),(x_2,y_2)) = \text{max} \{d(x_1,x_2),d(y_1,y_2)\}$.
\end{lem}

\begin{proof}
    We verify $G \times H$ acts by isometries. We compute 
\begin{align*}
d((x_1,y_1),(x_2,y_2)) &= \text{max}\{d(x_1,x_2),d(y_1,y_2)\}\\
 &= \text{max}\{d(g x_1,g x_2),d(h y_1,h y_2)\}\\
 &= d(( gx_1,hy_1),(g x_2,hy_2))\\
&= d((g,h) \cdot (x_1,y_1), (g,h) \cdot (x_2,y_2))
\end{align*} 
\end{proof}

We now recall some facts about Hom complexes, which have been heavily used to study graph homomorphisms and  chromatic number \cite{babson2006complexes, babson2007proof, dochtermann2009hom}.

\begin{defn}
    Let $G$ and $H$ be graphs. A  multihomomorphism from $G$ to $H$ is a map $\eta:V(G) \to 2^{V(H)}\setminus \emptyset$ such that $v_1 \sim v_2$ in $
    G$ implies for all $w_1 \in \eta(v_1)$ and $w_2 \in \eta(v_2)$, $w_1 \sim w_2$. We define a partial order on the set of multihomomorphisms from $G$ to $H$ by $\eta_1 \leq \eta_2$ if $\eta_1(v) \subset\eta_2(v)$ for all $v \in V(G)$. Define $Hom(G,H)$ to be this poset of multihomomorphisms from $G$ to $H$.
\end{defn}

Let $\Z_n$ denote the integers modulo $n$, and let $p$ always denote a prime. Let $C_p$ be the cycle on $p$ vertices. We identify the vertex set of $C_p$ with $\Z_p$, where $i \sim j$ if $i \pm 1 = j \pmod{p}$. Then $\Z_p$ acts on $Hom(C_p, G)$ with generator $\omega$ by $[\omega\cdot \eta](i) = \eta(i+1)$. This action is simplicial on $\Delta Hom(C_p, G)$. If $G$ is without loops, this action is free. With this action, if $f: G \to H$ is a graph homomorphism, then $f_*:Hom(C_p, G) \to_{\Z_p} Hom(C_p, H)$ defined by $f_*(\eta)(i) = f(\eta(i))$ is an equivariant poset map, and so induces an equivariant map on the corresponding order complexes.

We will need two lemmas guaranteeing the existence of equivariant maps. The first gives a sufficient topological condition.

\begin{lem}\label{6.2.2}
    \textnormal{(Lemma 6.2.2 of \cite{matouvsek2003using})} Let $X$ be a $(n-1)$-connected $G$ space and $K$ a finite, free simplicial $G$ complex of dimension at most $n$. Then $K \to_G X$.
\end{lem}

The second uses the language of representations. Recall a representation of a group $G$ is a group homomorphism $G \to GL(V)$ for some vector space $V$. We will abuse notation and refer to the representation as $V$ when no confusion will arise. We will only consider $V = \R^d$ and equip $\R^d$ with its usual inner product and require all representations to be orthogonal, meaning  that $G \to O(d)$ where $O(d)$ is the orthogonal group in dimension $d$. Under these assumptions, a representation restricts to a $G$ action on the unit sphere. For a representation $V$ the unit sphere with the action induced by the representation is denoted $S(V)$ and called the representation sphere of $V$. For a subgroup $H$ of $G$, let $V^H$ denote the points in $V$ fixed by all elements of $H$. A representation is called fixed point free if $V^G = (0)$.

\begin{thm}\label{thm4.4}
    \textnormal{(Theorem 4.4 of \cite{BasuGhosh2017}}) Let $V$ be a fixed point free $\Z_n$ representation such that for each prime power $p^k$ dividing $n$, $V^{\Z_{p^k}} \neq (0)$. Then for any $r$ there exists a $\Z_n$ representation $U$  such that $dim(U) > r$ and $S(U) \to_{\Z_n} S(V)$.
\end{thm}

We use these two lemmas to show there exists a $\Z_{2p}$ space $W_p$ such that all finite, free $\Z_{2p}$ complexes equivariantly map to $W_p$.

\begin{lem}\label{allmaptoZ}
     Let $\Z_{2p}$ act on $\R^3$ by the action generated by $$\alpha_p = \begin{pmatrix}
-1 & 0 & 0 \\ 
 0 & cos\left[\frac{(p-1)\pi}{p}\right] & -sin\left[ \frac{(p-1)\pi}{p} \right] \\
 0 & sin\left[\frac{(p-1)\pi}{p}\right] & cos \left[\frac{(p-1)\pi}{p}\right] \\
\end{pmatrix}.$$ 
Let $W_p$ denote the representation sphere of this action, equipped with the angular metric. Then, all finite free $\Z_{2p}$ complexes map equivariantly to $W_p$.
\end{lem}

\begin{proof}
    Let $K$ be a finite, free $\Z_{2p}$ complex of dimension $k$. Since the $\Z_2$ subgroup of $\Z_{2p}$ fixes all points of the form $(0,x_1,x_2)$ and the $\Z_p$ subgroup fixes all points of the form $(x_1,0,0)$, by \ref{thm4.4}  $S(U) \to_{\Z_{2p}} W_p$ where $S(U)$ is a representation sphere of dimension at least $k$. By \ref{6.2.2}, $K \to_{\Z_{2p}} S(U)$ and so composition gives $K \to_{\Z_{2p}} W_p$ as desired.
\end{proof}

\section{Construction of Counterexamples}

\begin{defn}
    Let $\Z_n$ act on $X$ by isometries with $\omega$ generating the action. Define $B_\omega(X,\ep)$ to be the graph with vertices the points of $X$, and $x_1 \sim x_2$ if $d(x_1, \omega^{-1}x_2) < \ep$ or $d(x_2, \omega^{-1}x_1) < \ep$.
\end{defn}

For $\Z_2$ acting on the sphere with the usual antipodal action, we recover the usual Borsuk graphs. Other generalizations of Borsuk graphs are defined and explored in \cite{martinez2024generalized, daneshpajouh2020dold}.

These generalized Borsuk graphs are functorial in the following way:

\begin{lem}\label{functor}
    Let $f: X \to_{\Z_n} Y$ satisfy that $d(x_1,x_2)< \delta \implies d(f(x_1),f(x_2))< \ep$ . Then $f_*: B_\omega(X, \delta) \to B_\omega(Y, \ep)$, defined by $f_*(x) = f(x)$, is a graph homomorphism.

    In particular, if $f$ is uniformly continuous, then for any $\ep > 0$ there is a $\delta > 0$ such that $f_*: B_\omega(X, \delta) \to B_\omega(Y, \ep)$ is a homomorphism.
\end{lem}

\begin{proof}
    Let $x_1 \sim x_2$ in $B_\omega(X, \delta)$. By potentially switching the names of $x_1$ and $x_2$ we may assume $d(x_1, \omega^{-1}x_2) < \delta$. This implies $d(f(x_1), f(\omega^{-1}x_2)) < \ep$. By equivariance, $f(\omega^{-1}x_2) = \omega^{-1}f(x_2)$, so $d(f(x_1), \omega^{-1}f(x_2)) = d(f(x_1), f(\omega^{-1}x_2)) < \ep$ and $f(x_1) \sim f(x_2)$ in $B_\omega(Y, \ep)$.
\end{proof}

For $p$ prime, we will lower bound the chromatic number of generalized Borsuk graphs using Hom complexes, for which we need the following lemma.

\begin{lem}\label{KtoHomK}
    Let $K$ be a polyhedron on which $\Z_p$ acts by isometries such that $K$ admits an equivariant triangulation. Let the action be generated by $\omega$.
    Then for every $\ep > 0$, $K \to_{\Z_p} \Delta Hom(C_p, B_\omega(K, \ep))$.
\end{lem}

\begin{proof}
    Take a fine enough equivariant triangulation of $K$ such that every simplex has diameter less than $\ep$. By abuse of notation, we will also call this simplicial complex $K$. We now define a poset map $\phi: \mathcal{F} (K) \to Hom(C_p, B_\omega(K, \ep))$ by $\phi(\sigma) = f_\sigma$, where $f_\sigma(i) = \omega^i\sigma$. Note that $\omega^i\sigma$ is a set of vertices in $B_\omega(K, \ep)$.

    We first verify that $f_\sigma$ is a multihomomorphism for all $\sigma$. Let $\sigma = \{k_1,k_2,...,k_n\}$. We want to show for any $i \in \Z_p$, $v \in f_\sigma(i)$, $w \in f_\sigma(i+1)$, we have $v \sim w$. There is some $k_j \in \sigma$ such that $v = \omega^i k_j$, and similarly $w = \omega^{i+1}k_\ell$.  Then $d(\omega^i k_j, \omega^{-1}\omega^{i+1}k_\ell) = d(\omega^i k_j, \omega^{i}k_\ell) = d(k_j, k_\ell)$ where the third equality holds as $\Z_p$ acts on $K$ by isometries. Since all simplices in $K$ have diameter less than $\ep$ and $k_j, k_\ell$ are in the same simplex, we have $d(k_j, k_\ell) < \ep$, so $d(\omega^i k_j, \omega^{-1}\omega^{i+1}k_\ell) < \ep$ and $v \sim w$. Thus $f_\sigma$ is a multihomomorphism.

    We now verify that $\phi$ is equivariant. Since $\omega$ generates the action, it suffices to show that $\omega \circ \phi = \phi \circ \omega$. We compute $\phi(\omega \sigma)(i)  =f_{\omega \sigma}(i) = \omega^i (\omega \sigma) = \omega^{i+1}\sigma = f_\sigma(i+1) = [\omega\cdot f_\sigma](i) = \omega(\phi(\sigma))(i)$, so $\phi \circ \omega = \omega \circ \phi$. 

    We now check that $\phi$ is order preserving. Let $\tau \subset \sigma$. Then $f\tau(i) = \omega^i\tau \subset \omega^i\sigma = f_\sigma(i)$, so $f_\tau \leq f_\sigma$.

    The map $\phi_*: \Delta \mathcal{F}(K) \to \Delta Hom(C_p, B_\omega(K, \ep))$ is then the desired equivariant map.
\end{proof}

We can now show that for carefully picked $\Z_p$ spaces $K$, the chromatic number of generalized Borsuk graphs can be arbitrarily large, generalizing the result that the usual Borsuk graphs on $S^n$ have chromatic number at least $n+2$.

\begin{lem}\label{chibigforp}
    Recall that for any $p$, $\Z_p$ acts freely on odd dimensional spheres by rotation, and this action is isometric with respect to the usual angular distance on $S^m$. Let this action be generated by $\omega$.

    For any $p$ and any $n$, there exists odd $m$ such that $\chi(B_\omega(S^m,\ep))>n$ for all $\ep>0$.
\end{lem}

\begin{proof}
    Let $p$ be any prime and $n$ be arbitrary. Let $k$ be the dimension of $Hom(C_p, K_n)$ and choose $m$ to be any odd integer strictly greater than $k$. Suppose for contradiction that $\chi((B_\omega(S^m,\ep))\leq n$ for some $\ep > 0$. Then there is a graph homomorphism $f:B_\omega(S^m,\ep) \to K_n $. By functoriality of the Hom complex, we get a map $f_*:Hom(C_p, B_\omega(S^m,\ep)) \to_{\Z_p} Hom(C_p, K_n)$. By \ref{KtoHomK}, we also have a map $\phi: S^m \to_{\Z_p} Hom(C_p, B_\omega(S^m,\ep))$. Composing, we get a map $f_*\circ \phi: S^m \to_{\Z_p} Hom(C_p, K_n)$. Since $S^m$ is $m-1\geq k-$connected and $Hom(C_p, K_n)$ is $k$ dimensional, this contradicts Dold's Theorem.
\end{proof}

While better lower bounds for $\chi(B_\omega(S^m,\ep))$ can be achieved with more sophisticated equivariant topology, the above suffices for our purposes.

We now prove a lemma relating the product of specific generalized Borsuk graphs to the generalized Borsuk graph of the product.

\begin{lem}\label{productlemma}
    Let $\Z_2$ act on $X$ by isometries with the action generated by $\nu$. Let $\Z_p$ act on $Y$ by isometries with the action generated by $\omega$. Endow $X \times Y$ with the product metric and product action as in \ref{metricproductaction}, and identify $\Z_2 \times \Z_p$ with $\Z_{2p}$ such that $(\nu,\omega)$ generates the $\Z_{2p}$ action. Then for all $\ep >0$, we have  
    $$\phi: B_\nu(X, \ep) \times B_\omega(Y, \ep) \to B_{(\nu,\omega)}(X \times Y, \ep) $$

    where $\phi(x,y) = (x,y)$ is a graph homomorphism.
\end{lem}

\begin{proof}
    Suppose $(x_1,y_1) \sim (x_2, y_2)$ in $B_\nu(X, \ep) \times B_\omega(Y, \ep)$. Then  we may assume  $d(x_1,\nu x_2) < \ep$ and $d(y_1,\omega^{-1}y_2) < \ep$, so
    
\begin{align*}
        d((x_1,y_1),(\nu,\omega)^{-1} \cdot (x_2,y_2)) &= d((x_1,y_1),(\nu^{-1}x_2,\omega^{-1}y_2))\\ &= \text{max} \{d(x_1,\nu x_2), d(y_1,\omega^{-1}y_2)\}\\
        &< \ep.
\end{align*}
 
    Thus $(x_1,y_1) \sim (x_2, y_2)$ in $B_{(\nu,\omega)}(X \times Y, \ep)$.
\end{proof}

Lastly, we need a universal upper bound on the chromatic number of generalized Borsuk graphs of free $\Z_{2p}$ spaces. Since all free $\Z_{2p}$ polyhedra map to $W_p$, as a first step we upper bound the chromatic number of $B_{\alpha_p}(W_p, \ep)$.

\begin{lem}\label{Zfinchi}
    Let $W_p$ be as in \ref{allmaptoZ}. For $p$ large enough, there exists $\ep>0$ such that $\chi(B_{\alpha_p}(W_p,\ep)) \leq 4$.
\end{lem}

\begin{proof}

It is well known that we may take $\ep>0$ small enough such that $\chi(B_\nu(S^2,\ep)) = 4$ where $\nu$ is the standard involution. Since $\alpha_p^{-1} \to \nu$ pointwise and $\alpha_p^{-1}$ and $\nu$ are linear, we have  $\alpha_p^{-1} \to \nu$ uniformly on $S^2$ as $p \to \infty$. Thus we may take $p$ such that $d( \alpha_p^{-1} x, \nu x) < \ep/2$ for all $x \in S^2$. We show that with this $p$, $B_{\alpha_p}(W_p, \ep/2)$ is a subgraph of $B_\nu(S^2, \ep)$. Let $x_1 \sim x_2$ in $B_{\alpha_p}(S^2, \ep/2)$, so we may assume $d(x_1, \alpha_p^{-1}x_2) < \ep/2$. Thus $d(x_1,\nu x_2) \leq d(x_1, \alpha_p^{-1}x_2) + d(\alpha_p^{-1}x_2, \nu x_2) < \ep/2 + \ep/2 = \ep$, so $x_1 \sim x_2$ in $B(S^2, \ep)$. Thus, $\chi(B_{\alpha_p}(W_p,\ep)) \leq 4$.

\end{proof}

We now prove our main result.

\begin{thm}
   For all $n$, there exists graphs $G$ and $H$ such that $\chi(G),\chi(H) > n$ and $\chi(G \times H) \leq 4$. 
\end{thm}

\begin{proof}
    Fix $p$ and $\ep$ as in \ref{Zfinchi}. It is then enough to show that for all $n$, there exists graphs $G$ and $H$ such that $\chi(G),\chi(H) > n$ and $G \times H \to B_{\alpha_p}(W_p,\ep)$. By \ref{chibigforp}, for our choice of $p$ and any $n$ we can take $m$ odd large enough that $\chi(B_\nu(S^m, \ep)), \chi(B_\omega(S^m, \ep)) > n$ for all $\ep > 0$ where $\nu$ is the usual antipodal action. By \ref{allmaptoZ}, there is some equivariant map $f: S^m \times S^m \to_{\Z_{2p}} W_p$. Since $S^m \times S^m$ is compact, $f$ is uniformly continuous, so by \ref{functor} there is some $\delta > 0$ such that $f_*:B_{(\nu,\omega)}(S^m\times S^m, \delta) \to B_{\alpha_p}(W_p, \ep)$ is a graph homomorphism. By \ref{productlemma}, we get a homomorphism $\phi: B_{\nu}(S^m, \delta) \times B_{\omega}(S^m, \delta) \to B_{(\nu,\omega)}(S^m\times S^m, \delta)$. Composing, we get $f \circ \phi: B_{\nu}(S^m, \delta) \times B_{\omega}(S^m, \delta) \to B_{\alpha_p}(W_p,\ep)$, completing the proof.
\end{proof}

\begin{rem}
    By the De Bruijn–Erd\Horig{o}s theorem \cite{bruijn1951colour}, one may take finite subgraphs $F_1 \subset B_\nu(S^m, \ep)$ and $F_2 \subset B_\omega(S^m, \ep)$ with $\chi(F_1) = \chi(B_\nu(S^m, \ep))$ and $\chi(F_2) = \chi(B_\omega(S^m, \ep))$ to get finite graphs showing that $f(n) \leq 4$.
\end{rem}

\section*{Acknowledgments}
We would like to thank Florian Frick and Matthew Kahle for comments on earlier drafts that greatly improved this manuscript. We also thank Tim Edwards for many thoughtful discussions.

\bibliographystyle{abbrv}
\bibliography{references}

\end{document}